\documentclass{amsart}
\usepackage{amssymb}

\usepackage[british,UKenglish,USenglish,american]{babel}
\usepackage{graphicx}
\usepackage{fancyhdr}
\usepackage{rotating}
\usepackage{amsmath}
\usepackage{amssymb}
\usepackage{amsthm}
\usepackage{tikz}
\usepackage{url}
\usepackage{enumerate}
\usepackage{mathtools}
\usepackage{setspace}
\usepackage{color}
\usepackage[normalem]{ulem}
\usetikzlibrary{positioning}

\newtheorem{innercustomthm}{Theorem}
\newenvironment{customthm}[1]
  {\renewcommand\theinnercustomthm{#1}\innercustomthm}
  {\endinnercustomthm}

  \newenvironment{proofclaim}{\begin{proof}[Proof of claim]}{\end{proof}}
  
  \newtheorem{innercustompr}{Proposition}
\newenvironment{custompr}[1]
  {\renewcommand\theinnercustompr{#1}\innercustompr}
  {\endinnercustompr}
  
    \newtheorem*{notation*}{Notation}

\newcommand{\db}{\overline{B}}
\newcommand{\du}{\overline{U}}
\newcommand{\dt}{\overline{T}}

\newtheorem{theorem}{Theorem}[section]
\newtheorem{lemma}[theorem]{Lemma}
\newtheorem{fact}[theorem]{Fact}
\newtheorem*{theorem*}{Theorem}
\newtheorem*{maintheorem*}{Main Theorem}
\newtheorem*{lemma*}{Lemma}
\newtheorem{corollary}[theorem]{Corollary}
\newtheorem{propo}[theorem]{Proposition}
\theoremstyle{definition}
\newtheorem{conjecture}[theorem]{Conjecture}
\newtheorem{definition}[theorem]{Definition}
\newtheorem{remark}[theorem]{Remark}
\newtheorem*{remark*}{Remark}

\newtheorem{question}[theorem]{Question}

\newtheorem*{claim*}{Claim}

\title[Small groups of finite Morley rank with a supertight automorphism]{Small groups of finite Morley rank with a supertight automorphism}

\author[U. Karhum\"{a}ki]{Ulla Karhum\"{a}ki}
\address{University of Helsinki}
\email{ukarhumaki@gmail.com}

\author[P. U\v{g}urlu Kowalski]{P\i nar U\u{g}urlu Kowalski}
\address{Istanbul Bilgi University}
\email{pinar.ugurlu@bilgi.edu.tr}

\thanks{The first author is funded by the Finnish Science Academy grant no: 338334. Part of this work was done when she was funded by the Finnish Science Academy grant no: 322795. The second author has completed part of this work during her sabbatical leave in 2022–2023 Academic Year which she spent at Wroc{\l}aw University. The presentation of the paper was finished during the authors' stay in the Mathematisches Forschungsinstitut Oberwolfach.}

\begin{document}
\maketitle

\begin{abstract} Let $G$ be an infinite simple group of finite Morley rank and of Pr\"{u}fer $2$-rank $1$ which admits a supertight automorphism $\alpha$ such that the fixed-point subgroup $C_G(\alpha^n)$ is pseudofinite for all integers $n > 0$. The main result of this paper is the identification of $G$ with ${\rm PGL}_2(K)$ for some algebraically closed field $K$ of characteristic $\neq 2$. \end{abstract}

\tableofcontents

\section{Introduction}\label{section:Intro} This paper is a step towards confirming the Cherlin-Zilber conjecture under some extra assumptions, which will be described later. Our framework is that of groups equipped with a rudimentary notion of dimension which is called the \emph{Morley rank}. Excellent general references for the topic are \cite{Borovik-Nesin, ABC}.

In the case of algebraic varieties one has the geometric notion of
\emph{Zariski dimension}. Morley rank generalises Zariski dimension and therefore the class of groups of finite Morley rank generalises the class of algebraic groups over algebraically closed fields. While the former class is strictly broader than the latter, these two classes are closely connected. This connection was highlighted by the famous Cherlin-Zilber Algebraicity conjecture proposed in the late 70's independently by Cherlin \cite{Cherlin1979} and Zilber \cite{Zilber1977}:

\begin{conjecture}[The Cherlin-Zilber conjecture] Infinite simple
groups of finite Morley rank are isomorphic to algebraic groups over algebraically closed fields. \end{conjecture}

The topic of groups of finite Morley rank lies in the border between
model theory and group theory and, from the group-theoretic point of view, the topic belongs somewhere in between algebraic group theory and finite group theory.

The classification of simple algebraic groups as Chevalley groups can be done by studying the maximal algebraic tori and their actions on unipotent subgroups; one then reveals the associated connected Dynkin diagram which determines the Lie type of the group in concern. The Classification of Finite Simple Groups (CFSG) states that, putting aside alternating groups and a handful of sporadic groups, a simple non-abelian finite group is a (twisted) group of Lie type. Again, one may identify the (twisted) group in concern by identifying the associated building.

Due to the (model-theoretically) tame nature of groups of finite Morley rank, from the wide tool-box developed in geometric stability theory, only few are useful in our context. Also, the first-order setting forbids the use of methods from algebraic geometry. As a result, in terms of techniques, the topic of infinite simple groups of finite Morley rank was pushed towards finite group theory. Indeed, so-far, the most fruitful approach towards the Cherlin-Zilber conjecture is the one commonly known as the \emph{Borovik programme}. In this approach, one borrows techniques from finite group theory and, using these techniques, aims towards `the analogue of CFSG in the context of infinite simple groups of finite Morley rank'. Below we briefly explain the greatest achievements of the Borovik programme and describe the current state of the Cherlin-Zilber conjecture. This then allows us to introduce and motivate our results and framework in satisfactory manners.

The Sylow $2$-theory of groups of finite Morley rank is both of great importance and well-understood. In a group of finite Morley rank $G$ the Sylow $2$-subgroups are conjugate and the \emph{connected component} $P^\circ$ of a Sylow $2$-subgroup $P$ of $G$ is well-defined and its structure is known. Depending on the structure of $P^\circ$, groups of finite Morley rank are split into four types: \emph{Even}, \emph{odd}, \emph{mixed} and \emph{degenerate }types (see Section~\ref{subsec:Sylow} for more detail). It is proven in \cite{ABC} that no infinite simple group of finite Morley rank of mixed type exists and the main result in \cite{ABC} states that even type infinite simple groups of finite Morley rank are isomorphic to Chevalley groups over algebraically closed fields of characteristic $2$. Also, in \cite{Borovik-Burdges-Cherlin}, it is proven that if a connected group of finite Morley rank is of degenerate type, then it contains no involutions. So, the unverified part of the Cherlin-Zilber conjecture proposes that no degenerate type infinite simple group of finite Morley rank exists and that an odd type infinite simple group of finite Morley rank is isomorphic to a Chevalley group over an algebraically closed field of characteristic $\neq 2$.

Consider a Chevalley group $G$ over an algebraically closed field $K$. The maximal (finite) number of copies of the \emph{Pr\"{u}fer $2$-group} $$\mathbb{Z}_{2^\infty}= \{x \in \mathbb{C}^{\times} : x^{2^n}=1 \, {\rm for} \,{\rm some} \,  n \in \mathbb{N} \}$$ in $G$ is called the \emph{Pr\"{u}fer $2$-rank} of $G$ and it is denoted by ${\rm pr}_2(G)$. One may measure the `size' of $G$ by its Pr\"{u}fer $2$-rank. The only simple Chevalley group over $K$ of Pr\"{u}fer $2$-rank $1$ is ${\rm PGL}_2(K)={\rm PSL}_2(K)$. One may also describe the `size' of $G$ by its Lie rank. The only simple Chevalley group of Lie rank $1$, over any field $F$ with $|F| > 2$, is ${\rm PSL}_2(F)$.

Recently Fr\'{e}con proved that an infinite simple group $G$ of Morley rank $3$ is isomorphic to ${\rm PGL}_2(K)$, where $K$ is an algebraically closed field \cite{Frecon2018}. While this is a groundbreaking result, there is very few hope that general results can be obtained by induction on the Morley rank. Instead, one could try to obtain general results by induction on the Pr\"{u}fer $2$-rank. One should start by trying to prove that if $G$ is an infinite simple group of finite Morley rank of ${\rm pr}_2(G)=1$ then $G$ is isomorphic to ${\rm PGL}_2(K)$. Since we are still far away from such an identification, different stronger assumptions are developed for `small' groups of finite Morley rank. For example, in the past, different authors have studied infinite simple groups of finite Morley rank in which every proper definable connected subgroup is solvable; such groups are called \emph{minimal simple}. While minimal simple groups of Pr\"{u}fer $2$-rank $1$ are rather successfully studied (\cite{Jaligot2000, CJ, Deloro2007, DeloroJaligot2016}), it is not known whether such a group is always isomorphic to ${\rm PGL}_2(K)$. This illustrates that identifying a `small' infinite simple group of finite Morley rank with ${\rm PGL}_2(K)$ is a very hard task.

In this paper, we adopt two different smallness assumptions (see Theorem~\ref{th:main} and Theorem~\ref{main-stringer} below). Our main result is the following:

\begin{theorem}\label{th:main} Let $G$ be an infinite simple group of finite Morley rank with ${\rm pr}_2(G)=1$  admitting a supertight automorphism $\alpha$. Assume that the fixed-point subgroup $C_G(\alpha^n)$ is pseudofinite for all $n\in \mathbb{N}\setminus \{0\}$. Then $G \cong {\rm PGL}_2(K)$, where $K$ is an algebraically closed field of characteristic $\neq 2$. \end{theorem}

Theorem~\ref{th:main} is obtained as a corollary of the following two results that we also prove. We wish to mention that Theorem~\ref{main-stringer} is formulated so that it can be directly applied in the ongoing work of the authors and Adrien Deloro \cite{Deloro-Karhumaki-Ugurlu}.

\begin{theorem}\label{main-stringer} Let $G$ be a connected group of finite Morley rank of odd type with a supertight automorphism $\alpha$. Assume that the following holds for all $n\in \mathbb{N}\setminus\{0\}$. \begin{enumerate}[(1)]
\item $P_n:=C_{G}(\alpha^n)$ is pseudofinite and for any non-trivial $ H_n\unlhd\unlhd P_n$, we have $C_{P_n}(H_n)=1$.
\item ${\rm Soc}(P_n) \cong {\rm PSL}_2(F_n)$ where  $F_n$ is a pseudofinite field of characteristic $\neq 2$. 
\item $\overline{{\rm Soc}(P_n)}=G$.
\end{enumerate} Then $G \cong {\rm PGL}_2(K)$, where $K$ is an algebraically closed field of characteristic $\neq 2$.\end{theorem}

\begin{propo}\label{propo:S} Let $G$ be a connected group of finite Morley rank of odd type with ${\rm pr}_2(G)=1$ and $H$ be a simple pseudofinite subgroup of $G$. Then $H\cong {\rm PSL}_2(F)$, where $F$ is a pseudofinite field of characteristic $\neq 2$. \end{propo}

Theorem~\ref{th:main} belongs to a project initiated by the second author in \cite{Ugurlu2009, Ugurlu2013} (and considered by the first author in \cite{karhumaki2019}). To explain this approach, we need to introduce more terminology and results.

An infinite structure (e.g. a group or a field) is called \emph{pseudofinite} if it satisfies every first-order property that is true in all finite structures. Pseudofinite fields were characterised in purely algebraic terms by Ax \cite{Ax1968}. While no such algebraic characterisation exists for pseudofinite groups, simple pseudofinite groups are classified as (twisted) Chevalley groups over pseudofinite fields \cite{Wilson1995}.

Algebraically closed fields with a generic automorphism were first studied by Macintyre in \cite{Macintyre1997}. Further, in \cite{CH1999}, Chatzidakis and Hrushovski axiomatised, and studied in depth, the first-order theory of algebraically closed fields with a generic automorphism. This theory is called ${\rm ACFA}$. It is known that if $(K,\sigma)$ is a model of ${\rm ACFA}$ then the fixed points ${\rm Fix}_{K}(\sigma)$ form a pseudofinite field (one should note that the axioms of pseudofinite fields are used to prove that ${\rm Fix}_{K}(\sigma)$ is pseudofinite).

In \cite{Hrushovski2002}, Hrushovski studied structures which satisfy certain nice model-theoretic properties and which have a generic automorphism. He proved that in this context the fixed-point subgroups of the generic automorphisms are pseudo-algebraically closed with small Galois groups. Further, he proved that any fixed-point subgroup arising this way admits a certain kind of measure which is similar to a non-standard probabilistic measure on pseudofinite groups. In the particular case of infinite simple groups of finite Morley rank, the aim is to prove that the fixed-points of a generic automorphism form a pseudofinite group. Indeed, in \cite{Ugurlu2009}, the following conjecture is formulated, from the results and observations of Hrushovski in \cite{Hrushovski2002}.

\begin{conjecture}\label{conj:principal}Let $G$ be an infinite simple group of finite Morley rank with a generic automorphism $\alpha$. Then the fixed-point subgroup $C_G(\alpha)$ is pseudofinite. \end{conjecture}

It follows from the results of Chatzidakis and Hrushovski \cite{CH1999} that the Cherlin-Zilber conjecture implies Conjecture~\ref{conj:principal}. There is an expectation that these two conjectures are actually equivalent and this expectation is supported by the results of the second author in \cite{Ugurlu2013}; she developed the following strategy towards proving the reverse implication. Below, we briefly introduce this strategy.

In order to work in a purely algebraic context, the second author considered a \emph{(super)tight} automorphism $\alpha$ (Definition~\ref{def:tight}) of an infinite simple group of finite Morley rank $G$. She proved that if $G$ admits $\alpha$ whose fixed-point subgroup $C_G(\alpha)$ is pseudofinite, then $C_G(\alpha)$ contains a simple pseudofinite (twisted) Chevalley subgroup $S$ so that $G$ is equal to the definable closure of $S$ in $G$. This gives hope for the identification of $G$ with an algebraic group over an algebraically closed field. Namely, to prove the expected equivalence between the Cherlin-Zilber conjecture and Conjecture~\ref{conj:principal}, it is enough to prove the following steps.

\begin{enumerate}[(i)]
\item \emph{Algebraic identification step}: Show that if the pseudofinite (twisted) Chevalley group $S$ has Lie rank $k$, then $G$ is isomorphic to a Chevalley group of Lie rank $k$ over an algebraically closed field $K$.
\item \emph{Model-theoretic step}: Prove that a generic automorphism of $G$ is supertight.
\end{enumerate} 

As our results ensure that ${\rm pr}_2(G)=1$ if and only if the Lie rank of $S$ is $1$, Theorem~\ref{th:main} proves the algebraic identification step in the case when $S$ is of Lie rank $1$. The case `$S$ is of Lie rank $1$' is one of the two crucial parts of the algebraic identification step; the other one being the case `$S$ is of Lie rank $2$'. If the analogue of Theorem~\ref{main-stringer} holds for Lie rank $2$ (i.e. replace `${\rm PSL}_2(F_n)$' in Theorem~\ref{main-stringer} by `a simple Chevalley group $X(F_n)$ of Lie rank $2$'), then there is a clear strategy for the proof of the algebraic identification step \cite{Deloro-Karhumaki-Ugurlu}.

This paper is organised as follows. In Section~\ref{definitions}, we give  background results that are needed in the proofs of Theorem~\ref{th:main}, Theorem~\ref{main-stringer} and Proposition~\ref{propo:S}. In Section~\ref{Sec:pseudofinite}, we prove Proposition~\ref{propo:S} and describe the structure of pseudofinite groups which satisfy some properties of the group $C_G(\alpha)$ in Theorem~\ref{main-stringer}. The definition (from \cite{Ugurlu2009}) of a supertight automorphism is given in Section~\ref{sec:tight_autom}. Then, in Section~\ref{Sec:proofs}, we prove Theorem~\ref{th:main} and Theorem~\ref{main-stringer}.

\section{Preliminaries}\label{definitions} 


\subsection{Automorphisms of Chevalley groups}\label{subsec:autom} We denote a Chevalley group of \emph{Lie type} $X$ over an arbitrary field $k$ by  $X(k)$, where $X$ comes from the list $A_n $, $B_n$, $C_n$, $D_n$, $E_6$, $E_7$, $E_8$, $F_4$, $G_2$; the subscript is called the \emph{Lie rank} of $X(k)$. If the Dynkin diagram has a non-trivial symmetry and the field $k$ satisfies suitable additional conditions, then $X(k)$ may be of \emph{twisted type}. We refer the reader unfamiliar with (twisted) Chevalley groups to \cite{Carter1971}.

There are four types of automorphisms of a (twisted) Chevalley group $X(k)$; called \emph{inner}, \emph{diagonal}, \emph{field} and \emph{graph} automorphism. We denote by ${\rm Aut}(X)$, ${\rm Inn}(X)$, ${\rm Diag}(X)$, ${\rm Aut}(k)$ and ${\rm Grp}(X)$ the group of all, inner, diagonal, field and graph automorphisms of $X(k)$, respectively.

\begin{theorem}\label{th:autom-finitegroups} Let $H=X(F)$ be a simple (twisted) Chevalley group over a perfect field $F$. Then \begin{enumerate}[(1)]
\item \emph{({\cite[Theorem 2.5.12]{GLS3}}).} ${\rm Aut}(X) =  {\rm Inn}(X){\rm Diag}(X)\rtimes {\rm Grp}(X) {\rm Aut}(F).$
\item \emph{(\cite{Steinberg1960}).} If $F$ is finite then $|{\rm Diag}(X)|$ and $| {\rm Grp}(X) |$ only depend on the Lie type $X$ of $H$.
\end{enumerate}\end{theorem}

\subsection{Pseudofinite groups}\label{subsec:pseudofinite}
Fix a language $\mathcal{L}$. A \emph{definable set} in an
$\mathcal{L}$-structure $\mathcal{M}$ is a subset $X\subseteq M^n$ which is the set of realisations of a first-order $\mathcal{L}$-formula $\varphi$. Throughout this paper, definable means definable possibly with parameters and we consider groups (resp. fields) as structures in the pure group (resp. field) language $\mathcal{L}_{gr}$. An \emph{$\mathcal{L}$-sentence} is an $\mathcal{L}$-formula in which all variables are bound by a quantifier. Two $\mathcal{L}$-structures $\mathcal{M}$ and $\mathcal{N}$ are \emph{elementarily equivalent}, denoted by $\mathcal{M} \equiv \mathcal{N}$, if they satisfy the same $\mathcal{L}$-sentences.

\begin{definition}A \emph{pseudofinite} group is an infinite group which satisfies every first-order sentence of $\mathcal{L}_{gr}$ that is true of all finite groups. \end{definition}

Let $\{ \mathcal{M}_i : i \in I \}$ be a family of
$\mathcal{L}$-structures and $\mathcal{U}$ be a non-principal
ultrafilter on $I$. We say that a property $P$
holds \emph{for almost all $i$} if $\{i : P \, \, {\rm holds \, \,
for}  \, \, \mathcal{M}_i \} \in \mathcal{U}$. Define $\sim_{\mathcal{U}}$ on $\prod_{i \in I} \mathcal{M}_i$ by $x \sim_{\mathcal{U}} y $ if and only if $\{i \in I : x(i)=y(i)\} \in \mathcal{U}$. The $\mathcal{L}$-structure
$\prod_{i \in I} \mathcal{M}_i /\sim_{\mathcal{U}}$ is denoted by $ \prod_{i \in I} \mathcal{M}_i / \mathcal{U}$ and called
the \emph{ultraproduct} of the $\mathcal{L}$-structures
$\mathcal{M}_i$ with respect to the ultrafilter $\mathcal{U}$. By \L o\'{s}'s Theorem, an infinite group (resp. $\mathcal{L}$-structure) is pseudofinite if and only if it is elementarily equivalent to a non-principal ultraproduct of finite groups (resp. $\mathcal{L}$-structures) of increasing orders, see \cite{Macpherson2018}.

Pseudofinite fields were axiomatised by Ax in \cite{Ax1968} as perfect pseudoalgebraically closed fields having exactly one extension of degree $n$ (in a fixed algebraic closure) for every positive integer $n$. As explained in \cite[Section 3]{Macpherson2018}, the following classification mostly follows from the work in \cite{Wilson1995}.

\begin{theorem}[Wilson \cite{Wilson1995}]\label{th:wilson} A simple group is
pseudofinite if and only if it is isomorphic to a (twisted)
Chevalley group over a pseudofinite field. \end{theorem}

\begin{remark}As one expects, CFSG is used in
Theorem~\ref{th:wilson}, which in turn is used in our proofs.\end{remark}

\subsection{Groups of finite Morley rank}\label{subsec:fRM} Groups of finite Morley rank are groups equipped with a notion of dimension which assigns to every definable set $X$ a dimension, called the \emph{Morley rank} and denoted by ${\rm rk}(X)$, satisfying well-known axioms given in \cite{Borovik-Nesin, ABC}. Below we list basic properties (also given in \cite{Borovik-Nesin, ABC}) of groups of finite Morley rank which the reader should bear in mind throughout the paper.

\begin{enumerate}[-]
\item Important examples of groups of finite Morley rank are algebraic groups over algebraically closed fields, where the Morley rank coincides with the Zariski dimension. Reader unfamiliar with the topic should systematically keep this example in mind.
\item A group of finite Morley rank $G$ do not have infinite descending chains of definable subgroups and in $G$ the length of any proper chain of centralisers is bounded.  
\item Using the chain condition above, we may define the \emph{connected component} $H^\circ$ of any subgroup $H$ of a group of finite Morley rank $G$ and the \emph{definable closure} $\overline{X}$ of any subset $X \subseteq G$: If $L \leqslant G$ is definable then $L^\circ$ is the intersection of definable subgroups of finite indices in $L$, $\overline{X}$ is the intersection of all definable subgroups of $G$ containing $X$, and, for any $H \leqslant G$, $H^\circ=H \cap \overline{H}^\circ$.
\end{enumerate}

\begin{fact}[{\cite[Lemma 2.15]{ABC}}]\label{fact:def_closure}
Let $G$ be a group of finite Morley rank and $X\subseteq G$. Then the following hold.\begin{enumerate}[(1)]
	\item If a subgroup $A \leqslant G$ normalises the set $X$, then $\overline{A}$ normalises $\overline{X}$.
	\item $C_G(X)= C_G(\overline{X})$.
	\item $\overline{N_G(X)} \leqslant N_G(\overline{X})$.
	\item For a subgroup $A \leqslant G$, $\overline{A^i}=\overline{A}^i$ and $\overline{A^{(i)}}=\overline{A}^{(i)}$.
	\item If  a subgroup $A \leqslant G$ is solvable (resp. nilpotent) of class $d$, then $\overline{A}$ is also solvable (resp. nilpotent) of class $d$. In particular, if $A$ is abelian then so is $\overline{A}$.
	\item \emph{(\cite[Lemma 5.35 (iii)]{Borovik-Nesin}).} Let $A \leqslant B \leqslant G$. If $A$ has finite index in $B$ then $\overline{A}$ has finite index in $\overline{B}$.
\end{enumerate}
\end{fact}

The rest of the basic results of the topic are not mentioned here, instead, specific references are provided at the technical moments. What really matters in our proofs are \emph{$2$-tori} and hence we next focus on the Sylow $2$-theory.

\subsubsection{Sylow $2$-theory}\label{subsec:Sylow} As mentioned in the introduction, the Sylow $2$-theory of groups of finite Morley rank is well-understood: In a group of finite Morley rank the Sylow $2$-subgroups are conjugate and the connected component $P^\circ$ of a Sylow $2$-subgroup $P$ of $G$ is the central product $U \ast T$  where $U$ is a \emph{$2$-unipotent} group (a definable and nilpotent $2$-group of bounded exponent) and $T$ is a \emph{$2$-torus} (a divisible abelian $2$-group) \cite[I.6]{ABC}. Groups of finite Morley rank can be split into four cases based on the structure of $P^\circ $ as follows:
\begin{enumerate}
\item \emph{Even type}:  $P^\circ$ is a non-trivial \emph{$2$-unipotent} group (that is $U\neq 1$ and $T=1$).
\item \emph{Odd type}:  $P^\circ$ is a non-trivial \emph{$2$-torus} (that is $U= 1$ and $T\neq 1$).
\item \emph{Mixed type}: $P^\circ$ is a central product of a non-trivial $2$-unipotent group and a non-trivial $2$-torus (that is $U\neq 1$ and $T \neq 1$).
\item \emph{Degenerate type}: $P^\circ$ is trivial (that is $U= 1$ and $T=1$).
\end{enumerate} 

If the ambient group $G$ is infinite and simple, then $G$ cannot be of mixed type \cite{ABC}. Also, in this case, either $P^\circ=1$ or $P^\circ$ is infinite \cite{Borovik-Burdges-Cherlin}. So an infinite simple group of finite Morley rank either contains no involutions or has infinite Sylow $2$-subgroups so that either $P^\circ=U$ or $P^\circ=T$. As even type infinite simple groups of finite Morley rank satisfy the Cherlin-Zilber conjecture \cite{ABC}, when trying to identify an infinite simple group of finite Morley rank $G$ with involutions, one may assume that $G$ is of odd type. Note that the connected component of a Sylow $2$-subgroup of a group of finite Morley rank of odd type is a direct product of finitely many copies of the Pr\"{u}fer $2$-group $\mathbb{Z}_{2^\infty}$.

The following results will be useful in our proofs.

\begin{fact}[Deloro and Jaligot {\cite[Proposition 27]{Deloro-Jaligot}}]\label{th:deloro-jaligot-sylow} Let $G$ be a connected group of finite Morley rank of odd type and with ${\rm pr}_2(G)=1$. Then there are exactly three possibilities for the isomorphism type of a Sylow $2$-subgroup $P$ of $G$.
\begin{enumerate}[(1)]
\item $P=P^\circ \cong \mathbb{Z}_{2^\infty}$.
\item $P=P^\circ \rtimes \langle w \rangle \cong \mathbb{Z}_{2^\infty} \rtimes \langle w \rangle$ for some involution $w$ which inverts $P^\circ$.
\item $P=P^\circ \cdot \langle w \rangle  \cong \mathbb{Z}_{2^\infty} \cdot \langle w \rangle$ for some element $w$ of order $4$ which inverts $P^\circ$. \end{enumerate}\end{fact}

\begin{theorem}[{\cite[Theorem 9.29]{Borovik-Nesin}}]\label{th:connected-solv-Sylow}Let $G$ be a connected solvable group of finite Morley rank. Then the Sylow $2$-subgroups are connected.\end{theorem}

\begin{fact}[{\cite[Lemma 10.3]{ABC}}] \label{fact:involutory_autom} Let $G$ be a connected group of finite Morley rank and $i$ be a definable involutory automorphism of $G$ with $C_G(i)$ finite. Then $G$ is abelian and $i$ inverts $G$. \end{fact}

\begin{theorem}[Deloro \cite{Deloro2009}]\label{th:centraliser_of_involution} Let $G$ be a connected odd type group of finite Morley rank and $i \in G$ be an involution. Then $C_G(i)/ C^\circ_G(i)$ has exponent dividing $2$.
\end{theorem}

\begin{theorem}[{Burdges and Cherlin \cite[Theorem $3^\ast$]{Burdges-Cherlin2009}}]\label{fact:centraliser-Sylow-2}Let $G$ be a connected odd type group of finite Morley rank and $i\in G$ be an involution. Then $i$ belongs to a $2$-torus of $G$.\end{theorem}

\begin{theorem}[Alt\i nel and Burdges {\cite[Theorem 1]{Altinel-Burdges}}]\label{th:good_torus_centraliser} If $T$ is a $2$-torus of a connected group $G$ of finite Morley rank then $C_G(T)$ is connected. \end{theorem}

\section{Results on pseudofinite groups}\label{Sec:pseudofinite}In this section we prove those results on pseudofinite groups which are needed in the proofs of Theorem~\ref{th:main} and Theorem~\ref{main-stringer}.

\subsection{A structural result on certain pseudofinite groups}\label{sec:socle} Let $G$ be a group and $k$ be a positive integer. We say that $G$ is of \emph{centraliser dimension} $k$ if the longest proper descending chain of centralisers in $G$ has length $k$. If such an integer $k$ exists, then $G$ is called a group of \emph{finite centraliser dimension}.

The \emph{socle} of a group $G$, denoted by ${\rm Soc}(G)$, is the subgroup generated by all minimal normal (not necessarily proper) subgroups of $G$.

A subgroup $H$ of a group $G$ is \emph{subnormal} if there is a finite ascending chain of subgroups starting from $H$ and ending at $G$, so that each is a normal subgroup of its successor. This is denoted by $H \unlhd\unlhd G$.

Using similar arguments as in \cite{Ugurlu2013} and \cite{Karhumaki2021}, we obtain the following result.

\begin{propo}\label{prop:structure-pf} Let $G\equiv \prod_{i\in I} G_i/\mathcal{U}$ be a pseudofinite group of finite centraliser dimension. If for any non-trivial $H  \unlhd\unlhd  G$ we have $C_G(H)=1$, then the following hold. \begin{enumerate}[(a)]
\item \label{socle} For almost all $i$, ${\rm Soc}(G_i)$ is non-abelian simple and $C_{G_i}({\rm Soc}(G_i)) = 1$.
\item \label{socleU}$\prod_{i\in I}{\rm Soc}(G_i)/ \mathcal{U}$ is an infinite definable normal subgroup of $\prod_{i\in I}{G_i}/ \mathcal{U}$ which is isomorphic to a simple (twisted) Chevalley group $X(F)$ where $F$ is a pseudofinite field.
\item \label{socleG} ${\rm Soc}(G) \equiv \prod_{i\in I}{\rm Soc}(G_i)/ \mathcal{U}$ and hence ${\rm Soc}(G) \cong X(F)$.
\item \label{abelian-by-finite} $G/{\rm Soc}(G)$ is abelian-by-finite and $G \leqslant {\rm Soc}(G){\rm Diag}(X)\rtimes {\rm Graph}(X)A$, where $A \leqslant {\rm Aut}(F)$ is abelian-by-finite. If $X={\rm PSL}_2$, then $A$ is abelian.
\end{enumerate}
\end{propo}
\begin{proof} (\ref{socle}) (cf. \cite{Ugurlu2013}.) For any group the following property \begin{center} ``The centraliser of any non-trivial normal subgroup is trivial"
\end{center} can be expressed by the first-order sentence
$$\forall x \forall y [(x \neq 1 \wedge y \neq 1)\to \exists z [y, x^z] \neq 1].$$
By assumption, the sentence above holds in $G$, so using \L o\'{s}'s Theorem, it holds in $G_i$ for almost all $i$. It immediately follows that $C_{G_i}({\rm Soc}(G_i)) = 1$ and ${\rm Soc}(G_i)$ is non-abelian. The simplicity of ${\rm Soc}(G_i)$ follows exactly as in \cite{Ugurlu2013}. We will briefly sketch this argument here. The finite group $G_i$ has at least one minimal non-trivial normal subgroup. By the first-order property above, it has a unique such, which is the socle. Such subgroup is clearly characteristically simple. To see that it is actually simple, we need to use two assumptions on $G$ (finite centraliser dimension and triviality of the centralisers of subnormal subgroups). Since this part is a bit technical and already proven before, we simply refer the reader to \cite[Lemma 3.10 and Lemma 3.11]{Ugurlu2013}.

(\ref{socleU}) (cf. \cite{Ugurlu2013})
Clearly, $\prod_{i\in I}{\rm Soc}(G_i)/ \mathcal{U}$ is a  normal subgroup of $\prod_{i\in I}{G_i}/ \mathcal{U}$. The following holds for almost all $i$: If there was a bound on the sizes $|{\rm Soc}(G_i)|$ then there would be a bound on the sizes $|G_i|$ as $G_i \leqslant {\rm Aut}({\rm Soc}(G_i))$, contradicting the fact that $G$ is infinite. As the finite centraliser dimension property forbids ${\rm Soc}(G_i)$'s to be bigger and bigger alternating groups, by CFSG, ${\rm Soc}(G_i) \cong X(\mathbb{F}_{i})$ where $X(\mathbb{F}_{i})$ is a simple (twisted) Chevalley group with $|\mathbb{F}_{i}|>8$. Also, the groups ${\rm Soc}(G_i)$ are of the same Lie rank by the finite centraliser dimension; hence they are all of the same Lie type. As a result, $\prod_{i\in I}{\rm Soc}(G_i)/ \mathcal{U}\cong X(F)$ for some pseudofinite field $F$.  For definability, we refer the reader again to \cite{Ugurlu2013}. Here is a brief argument. Since ${\rm Soc}(G_i) \cong X(\mathbb{F}_{i})$ where  $|\mathbb{F}_{i}|>8$, by the result of Ellers and Gordeev in \cite{Ellers-Gordeev1998}, there is $x_{i} \in {\rm Soc}(G_i)$ so that $${\rm Soc}(G_i)=x_{i}^{{\rm Soc}(G_i)}x_{i}^{{\rm Soc}(G_i)}.$$ This allows us to conclude that $\prod_{i\in I}{\rm Soc}(G_i)/ \mathcal{U}$ is definable in $\prod_{i\in I}{G_i}/ \mathcal{U}$ (for detail, see \cite[Lemma 3.12]{Ugurlu2013}).

(\ref{socleG}) By (b) there is $S \leqslant G$ so that $S\equiv \prod_{i\in I}{\rm Soc}(G_i)/ \mathcal{U}$; hence $S\cong X(F)$ where $F$ is a pseudofinite field. Being simple, $S$ is a minimal normal subgroup of $G$ and, by assumption, $C_G(S)=1$. So $S$ is the unique minimal normal subgroup of $G$ and $S={\rm Soc}(G)$.

(\ref{abelian-by-finite})  By
Theorem~\ref{th:autom-finitegroups}, we have
$$ G_i \leqslant {\rm Aut}({\rm Soc}(G_{i})) \cong
\left(X(\mathbb{F}_{i}) \cdot {\rm Diag}(X) \right) \rtimes\left( {\rm
Grp}(X)\cdot {\rm Aut}(\mathbb{F}_{i}) \right),$$ where the groups
${\rm Diag}(X),{\rm Grp}(X)$ are finite and do not depend on $i$. Let us set:
$$C:= \prod_{i\in I}{\rm Aut}(\mathbb{F}_{i})/\mathcal{U},\
\mathbb{F}:= \prod_{i\in I}\mathbb{F}_{i}/\mathcal{U},\
H:=\prod_{i\in I}G_{i}/\mathcal{U},\ S_H:=X(\mathbb{F}).$$
It is now clear that
$$G \equiv H \leqslant \prod_{i\in I} {\rm
Aut}(X(\mathbb{F}_{i}))/\mathcal{U}\equiv \left( S_H \cdot {\rm
Diag}(X) \right) \rtimes\left( {\rm Grp}(X))\cdot C \right),$$ where $C$ is abelian. Now
$$G/{\rm Soc}(G) \equiv H/S_H  \leqslant {\rm Diag}(X) \rtimes\left( {\rm
Grp}(X)\cdot C \right)$$
and the group $H/S_H$ is abelian-by-$d$, where $d$ depends on the sizes $|{\rm Diag}(X)|$ and $|{\rm Grp}(X)|$. Since this is a first-order property, $G/{\rm Soc}(G)$ is abelian-by-$d$ as well. On the other hand, again by Theorem~\ref{th:autom-finitegroups}, we have $$G  \leqslant {\rm Aut}({\rm Soc}(G)) \cong
\left({\rm Soc}(G) \cdot {\rm Diag}(X) \right) \rtimes\left( {\rm
Grp}(X)\cdot {\rm Aut}(F) \right).$$ Consider the following projection map $$\Pi: {\rm Soc}(G) \rtimes {\rm Aut}(F) \longmapsto  {\rm Aut}(F)$$ and set $A:= \Pi(G)$. Since ${\rm Soc}(G) \leqslant {\rm Ker}(\Pi)$, by above, $A$ is abelian-by-finite. If $X= {\rm PSL}_2$, then $$G/{\rm Soc}(G) \equiv H/S_H  \leqslant {\rm Diag}({\rm PSL}_2) \rtimes C. $$ Since ${\rm Diag}({\rm PSL}_2) \cong C_2$, the semidirect product above is a direct product and hence $H/S_H$ is abelian; and so is $A$.  
\end{proof}


\subsection{Proof of Proposition~\ref{propo:S}} \label{subsec:S} In this section, we shall prove Proposition~\ref{propo:S}. For the proof, we will need the following results from finite group theory.

\begin{theorem}\label{th:finite-groups}Let $H$ be a finite group. Then the following hold. \begin{enumerate}[(1)]
\item \emph{(\cite[Theorem 12.7]{Passman1968}).} If $H$ has a cyclic Sylow $2$-subgroup then $H$ is not simple.
\item \emph{(\cite{Brauers}).} If $H$ has a generalised quaternion Sylow $2$-subgroup then $H$ is not simple.
\item \emph{(\cite{Gorenstein-Walter}).} If $H$ is simple and has a dihedral Sylow $2$-subgroup, then $H$ is isomorphic either to ${\rm PSL}_2(\mathbb{F}_q)$, for $q \geqslant 5$ odd, or to the alternating group $A_7$.
\end{enumerate}\end{theorem}

We include the proofs of the following two facts  here since we could not find exact references.

\begin{fact}\label{lemma:gen-qua-sylow} Let $H$ be a finite group in which every subgroup of order $4$ is a cyclic group. Then $H$ is not simple.\end{fact}
\begin{proof}Towards a contradiction, assume that $H$ is simple. Let $P$ be a Sylow $2$-subgroup of $H$. Then $P$ contains a unique involution, so it is cyclic or generalised quaternion. By Theorem~\ref{th:finite-groups}, $H$ is not simple.\end{proof}

\begin{fact}\label{lemma-subgroups-cyclic-dih} Let $H$ be a finite $2$-group such that every subgroup of order $8$ in $H$ is either a cyclic or a dihedral group. Then $H$ is either a cyclic or a dihedral group.
\end{fact}

\begin{proof} Let $|H|=2^m$. We may assume that $m >3$. We first observe that $H$ contains no normal Klein $4$-group: Assume contrary. Then $E \unlhd H$ is a Klein $4$-group.  Since ${\rm Aut}(E) = {\rm GL}_2(\mathbb{F}_2) \cong {\rm Sym}_3$ and $H/C_H(E)$ embeds in ${\rm Aut}(E)$ we have $[H: C_{H}(E)] \leqslant 2$. Therefore, $|C_{H}(E)| \geqslant 8$. If $C_{H}(E)$ contains an element $h$ of order $4$, then $E\langle h\rangle$ is a non-cyclic abelian subgroup of order $16$ or $8$. But $E\langle h\rangle$ cannot be of order $8$ by assumption and, if $E\langle h\rangle$ is of order $16$, then it is forced to be isomorphic to the group $\mathbb{Z}_4 \times \mathbb{Z}_2 \times \mathbb{Z}_2$ which contains a non-cyclic abelian subgroup of order $8$, contradictory to assumption. Hence, $C_{H}(E)$ is an elementary abelian $2$-subgroup of order $\geqslant 8$. This contradiction proves that $H$ contains no normal Klein $4$-group. So $H$ is one of the following groups (\cite[Ex. 9, Chapter 5]{gorenstein1980}): a cyclic, a dihedral, a semidihedral, or a generalised quaternion group. As the latter two contain a quaternion subgroup $Q_8$ (\cite[Theorem 4.3]{gorenstein1980}), the claim holds. \end{proof}

\begin{custompr}{1.4} Let $G$ be a connected group of finite Morley rank of odd type with ${\rm pr}_2(G)=1$ and $H$ be a simple pseudofinite subgroup of $G$. Then $H\cong {\rm PSL}_2(F)$, where $F$ is a pseudofinite field of characteristic $\neq 2$.
\end{custompr}

\begin{proof}By Theorem~\ref{th:wilson}, $H \cong \prod_{i\in I} (X(\mathbb{F}_{i}))_i / \mathcal{U}$, where $\mathcal{U}$ is a non-principal ultrafilter on $I$ and, for almost all $i$, $X(\mathbb{F}_{i})$ is a finite simple (twisted) Chevalley group.  Let $P_H \leqslant H$ be any $2$-subgroup and let $P$ be a Sylow 2-subgroup of $G$ such that  $P_H \leqslant P \leqslant G$. Note that $G$ has conjugate Sylow $2$-subgroups and their structure is given  in Theorem~\ref{th:deloro-jaligot-sylow}. If $P$ has type (1) or (3) from Theorem~\ref{th:deloro-jaligot-sylow}, then $G$, and therefore $H$, contains no Klein $4$-group. This is a first-order property, which therefore holds in almost all of the $X(\mathbb{F}_{i})$'s, contradicting Fact~\ref{lemma:gen-qua-sylow}. It follows that $P$ has type (2) from Theorem~\ref{th:deloro-jaligot-sylow}. Thus every subgroup of order $8$ of $G$, and of $H$, is cyclic or dihedral. This is another first-order property, which holds in almost all of the $X(\mathbb{F}_{i})$'s. Thus, by Fact~\ref{lemma-subgroups-cyclic-dih}, Sylow $2$-subgroups of $X(\mathbb{F}_{i})$ are cyclic or dihedral; by Fact~\ref{lemma:gen-qua-sylow} they are dihedral (note that if the Sylow $2$-subgroups of $H$ are of order $4$ then they must be dihedral). Now, as $H$ is infinite, Theorem~\ref{th:finite-groups} gives the result.\end{proof}
 
\section{Supertight automorphisms}\label{sec:tight_autom}
Recall that the definable closure of any subset $X$ in a group of finite Morley rank $G$ is denoted by $\overline{X}$.

\begin{definition}\label{def:tight} An automorphism $\alpha$ of an infinite connected group of finite Morley rank $G$ is called \emph{tight} if, for any connected definable and set-wise $\alpha$-invariant subgroup $H$ of $G$, we have $\overline{C_{H}(\alpha)}=H.$ Further, $\alpha$ is called \emph{supertight} if both of the following hold. \begin{enumerate}
\item For each $n\in \mathbb{N} \setminus \{0\}$ the power $\alpha^n$ is a tight automorphism of $G$.
\item \label{fixed-points} For any $m,n\in \mathbb{N} \setminus \{0\}$, if $m|n$, then $C_G(\alpha^m)<  C_G(\alpha^n)$.
\end{enumerate} 
\end{definition}

The notion of a (super)tight automorphism is defined so that one mimics the situation in which $G$ is a simple Chevalley group over an algebraically closed field $K$ and $\alpha$ is a generic automorphism of $K$: If $G$ is a simple Chevalley group, recognised as the group of $K$-rational points of an algebraically closed field $K =\prod_{p_i \in P} \mathbb{F}^{alg}_{p_i}/\mathcal{U}$, then the \emph{non-standard Frobenius automorphism} $\phi_{\mathcal{U}}$ of $K$ induces on $G$ a supertight automorphism. Here, $P$ is the set of all prime numbers, $\mathcal{U}$ is a non-principal ultrafilter on $P$ and $\phi_{\mathcal{U}}$ is the map from $K$ to $K$ sending an element $[x_i]_{\mathcal{U}}$ to the element $[x_i^{p_i}]_{\mathcal{U}}$. In this situation, for all $n\in \mathbb{N}\setminus \{0\}$, the fixed-point subgroup $C_G(\phi_{\mathcal{U}}^n)$, which is equal to $G({\rm Fix}_{\phi_{\mathcal{U}}^n}(K))$, is clearly a pseudofinite group and, if $m|n$, then $C_G(\phi_{\mathcal{U}}^m) < C_G(\phi_{\mathcal{U}}^n)$. Further, the pair $(K,\phi_{\mathcal{U}}^n)$ is a model of ${\rm ACFA}$ \cite{Hrushovski2004}. Due to this example, we expect that the following question has a positive answer.

\begin{question} Is  a tight automorphism of an infinite simple group of finite Morley rank supertight?
\end{question}

We will use the following results in the proof of Theorem~\ref{th:main}. 

\begin{theorem}[U\u{g}urlu {\cite[Theorem 3.1]{Ugurlu2013}}]\label{th:ugurlu} Let $G$ be an infinite simple group of finite Morley rank and $\alpha$ be a tight automorphism of $G$. Assume that the fixed-point subgroup $C_G(\alpha)=P \equiv \prod_{i\in I} P_i/\mathcal{U}$ is pseudofinite. Then $P$ has a definable normal subgroup  $S$ such that $S \equiv \prod_{i\in I }{\rm Soc}(P_i)/ \mathcal{U}$ is a simple pseudofinite group.\end{theorem}

\begin{fact}[U\u{g}urlu {\cite[Lemma 3.3]{Ugurlu2013}}]\label{fact:ugurlu}Let $G$ be an infinite simple group of finite Morley rank with a supertight automorphism $\alpha$. Assume that the fixed-point subgroup $C_G(\alpha^n)=P_n$ is pseudofinite for all $n\in \mathbb{N}\setminus \{0\}$. Then for any non-trivial $ H_n \unlhd\unlhd P_n$, we have  $\overline{H_n}=G$ and $C_{P_n}(H_n)=1$.\end{fact}

Let $G$ be an infinite simple group of finite Morley rank with a tight automorphism $\alpha$ whose fixed-point subgroup is pseudofinite. We finish this section by noting that $G$ cannot be of degenerate type. Assume contrary. Then $G$ has no involutions and neither does the simple pseudofinite group $ S  \leqslant G$ (Theorem~\ref{th:ugurlu}).  We have $S \cong \prod_{i\in I} (X(\mathbb{F}_i))_i/\mathcal{U}$ where, for almost all $i$, $X(\mathbb{F}_i)$ is a finite (twisted) Chevalley group (Theorem~\ref{th:wilson}).  Since having no involutions is a first-order property, $X(\mathbb{F}_i)$ has no involutions contradictory to the famous Feit-Thompson odd order theorem.

\section{Proof of the main theorem}\label{Sec:proofs} In this section we prove Theorem~\ref{th:main} and Theorem~\ref{main-stringer}:

\begin{customthm}{1.2}Let $G$ be an infinite simple group of finite Morley rank with ${\rm pr}_2(G)=1$  admitting a supertight automorphism $\alpha$. Assume that the fixed-point subgroup $C_G(\alpha^n)$ is pseudofinite for all $n\in \mathbb{N}\setminus \{0\}$. Then $G \cong {\rm PGL}_2(K)$, where $K$ is an algebraically closed field of characteristic $\neq 2$.
\end{customthm}

\begin{customthm}{1.3}Let $G$ be a connected group of finite Morley rank of odd type with a supertight automorphism $\alpha$. Assume that the following holds for all $n\in \mathbb{N}\setminus\{0\}$. \begin{enumerate}[(1)]
\item $P_n:=C_{G}(\alpha^n)$ is pseudofinite and for any non-trivial $ H_n\unlhd\unlhd P_n$, we have $C_{P_n}(H_n)=1$.
\item ${\rm Soc}(P_n)\cong {\rm PSL}_2(F_n)$ where $F_n$ is a pseudofinite field of characteristic $\neq 2$. 
\item $\overline{{\rm Soc}(P_n)}=G$.
\end{enumerate} Then $G \cong {\rm PGL}_2(K)$, where $K$ is an algebraically closed field of characteristic $\neq 2$.\end{customthm}

We first observe that Theorem~\ref{th:main} follows from Theorem~\ref{main-stringer}: \begin{proof}[Proof of Theorem~\ref{th:main} if Theorem~\ref{main-stringer} holds.] Let $G$ be as in Theorem~\ref{th:main}. Then $G$ is connected and of odd type. Set $P_n:=C_{G}(\alpha^n)$. Assumptions (1) and (3) of Theorem~\ref{main-stringer} hold by Fact~\ref{fact:ugurlu}. Also, as $P_n \leqslant G$  is of finite centraliser dimension, Proposition~\ref{prop:structure-pf} and Proposition~\ref{propo:S} imply that (2) of Theorem~\ref{main-stringer} holds. So, Theorem~\ref{main-stringer} implies that $G \cong {\rm PGL}_2(K)$, where $K$ is an algebraically closed field of characteristic $\neq 2$. \end{proof}

In what follows, we shall prove Theorem~\ref{main-stringer} step-by-step. First, we fix the set-up and notation for the rest of the paper and explain the overall strategy of our proof.

\subsection{Notation, set-up and overall strategy}\label{notation}Let $F$ be a pseudofinite field of characteristic $\neq 2$. We now fix notation on classical subgroups of ${\rm PGL}_2(F)$ and collect some well-known facts about them. Firstly, bear in mind the following:
$$\mathrm{Aut}(\mathrm{PSL}_2(F))\cong
\mathrm{Aut}(\mathrm{PGL}_2(F))\cong \mathrm{PGL}_2(F)\rtimes
\mathrm{Aut}(F).$$ Let $U_2(F), D_2(F) \leqslant {\rm GL}_2(F)$ denote the groups of upper unitriangular matrices and of diagonal matrices, respectively. Then $U_2(F)\cong F^+$, $D_2(F)\cong F^{\times}\times F^{\times}$ and $D_2(F)\cap \mathrm{SL}_2(F)\cong F^{\times},$ where $F^+$
and $F^{\times}$ denote the additive and multiplicative groups of $F$,
respectively. Notice that $[F^\times:(F^\times)^2]=2$ as $F$ is a pseudofinite field of characteristic $\neq 2$. Let $\pi:\mathrm{GL}_2(F)\to \mathrm{PGL}_2(F).$

\begin{enumerate}[(a)]
\item \label{PGL2-unipotent} We denote by $U$ (a \emph{unipotent subgroup} of ${\rm
PGL}_2(F)$) the image of $U_2(F)$ by
$\pi$ in ${\rm PGL}_2(F)$. Clearly $U\cong F^+$.

\item \label{PGL2-torus} We denote by $T $ (a \emph{maximal
split torus} of ${\rm PGL}_2(F)$) the image of
$D_2(F)$ by $\pi$ in ${\rm PGL}_2(F)$. Clearly $T\cong F^\times$. The unique involution of $T$ is denoted by $i$ (note that $\text{char}(F)\neq 2$).

\item \label{PGL2-borel} A \emph{Borel subgroup} $B=U\rtimes T$ of ${\rm PGL}_2(F)$ is the image by $\pi$ in ${\rm PGL}_2(F)$ of the group of upper-triangular matrices.

\item \label{PGL2-NU-CU}We have $N_{{\rm
PGL}_2(F)}(B)=N_{{\rm PGL}_2(F)}(U)=B$, with $i$
acting on $U$ by inversion, and $C_{{\rm PGL}_2(F)}(u)=U$ for all $u \in U \setminus \{0\}$.

\item \label{PGL2-minimalU}The unipotent group $U$ has no proper non-trivial $T$-normal subgroups.

\item \label{PGL2-Weyl} The \emph{Weyl involution} $w$ is defined as
the image by $\pi$ in ${\rm PGL}_2(F)$ of the matrix ${\begin{pmatrix}
0 & -1 \\ 1 & 0 \end{pmatrix}}$. The Weyl involution $w$ inverts
$T$ and $ N_{{\rm PGL}_2(F)}(T)/T=
\langle wT \rangle.$

\item \label{PGL2-NT-CT-Ci}We have $C_{{\rm PGL}_2(F)}(i)=N_{{\rm
PGL}_2(F)}(T)=\langle w, T\rangle$ and $C_{{\rm
PGL}_2(F)}(t)=T$ for all $t \in T$ so that $t^2\neq
1$. Also $T\cap T^{g}=1$ for all $g\in {\rm
PGL}_2(F)\setminus N_{{\rm PGL}_2(F)}(T)$.

\item \label{PGL2-disjoint-union}We have
${\rm PGL}_2(F)=B\sqcup U
w B$, where $\sqcup$ stands for the disjoint union.\end{enumerate}

A $2$-transitive permutation group in which a stabiliser of any three distinct points is the identity is called a \emph{Zassenhaus group}. A Zassenhaus group is said to be \emph{split} if a one-point stabiliser $B$ is equal to $U \rtimes T$ with $T < B$ and $T\cap T^b=1$ for all $b \in B\setminus T$. We shall identify the group $G$ as in Theorem~\ref{main-stringer} by invoking the following classical result.

\begin{theorem}[Delahan and Nesin \cite{DelahanNesin}]\label{th:DN} Let $G$ be an infinite split Zassenhaus group of finite Morley rank. If the stabiliser of two distinct points contains an involution, then \mbox{$G \cong {\rm PGL}_2(K)$} for some algebraically closed field $K$ of characteristic $\neq 2$. \end{theorem}

From now on, we use the following notation.
\begin{notation*}\begin{enumerate}[(i)]\item Let $G$ and $\alpha$ be as in Theorem~\ref{main-stringer}. We have $C_G(\alpha^n)=P_n$ and ${\rm Soc}(P_n) \cong {\rm PSL}_2(F_n)$.
\item In the course of the proof, we shall freely replace $\alpha$ by a power, and therefore, mostly, we omit subscript $n$. Thus, unless we need to specify a power of $\alpha$, we simply write $P$, ${\rm Soc}(P)$ and ${\rm PGL}_2(F)$ instead of $P_n$, ${\rm Soc}(P_n)$ and ${\rm PGL}_2(F_n)$.
\item We use the notation above for subgroups and elements of ${\rm PGL}_2(F)$: $U$, $T$ and $B$ stand for the unipotent subgroup, the maximal split torus and the Borel subgroup, $i$ is the unique involution of $T$ and $w$ is the Weyl involution inverting $T$. 
\item We shall start our proof by showing that ${\rm PGL}_2(F) \leqslant P$ (Corollary~\ref{corol:PGL2inP2}). Then, by Proposition~\ref{prop:structure-pf}, we have $P ={\rm PGL}_2(F) \rtimes A,$ where $A$ is an abelian subgroup of ${\rm Aut}(F)$. 
\end{enumerate}
\end{notation*}

As noted above, we start  by proving that ${\rm PGL}_2(F) \leqslant G$. This is an important step as it ensures that the involution $i\in T$ is an element of $G$. Then, as $ \overline{{\rm PGL}_2(F)}=G$, it suffices to prove that the definable closures in $G$ of subgroups of ${\rm PGL}_2(F)$ behave `as one would expect'. That is, we prove that $G$ is a split Zassenhaus group, acting on the set of left cosets of $\db$ in $G$, with a one-point stabiliser $\db$ and a two-point stabiliser $\dt$ by studying the definable closures in $G$ of subgroups of ${\rm PGL}_2(F)$. The three main steps of the proof are: \begin{enumerate}
\item ${\rm PGL}_2(F) \leqslant P$ (Section~\ref{subsec:PGL2-in_P}),
\item ${\rm PGL}_2(F)=P$ (Section~\ref{subsec:P/S-finite}), and
\item $\db \cap \du^g=1$ for all $g \in G\setminus \db$. (Section~\ref{subsec:normalisers}).
\end{enumerate} In Section~\ref{subsec:structures} we make some preliminary observations about the groups $\dt$ and $\du$ that are used in steps (2) and (3). After having (1)-(3), the identification of $G$ becomes standard; at this point we may use known techniques (\cite{Jaligot2000, CJ, DeloroJaligot2016}). To keep the text self-contained, we give this final identification in Section~\ref{subsec:final-identification}.

Throughout the rest of the paper, we repeatedly use facts (\ref{PGL2-unipotent})--(\ref{PGL2-disjoint-union}) collected above. The following remark will also be used repeatedly, without referring to it again.

\begin{remark}\label{lemma:def_closure_stabilised_by_alpha} Let $X \subseteq G$ be set-wise $\alpha$-invariant, then so is $\overline{X} \leqslant G$. This is obvious since $\alpha$ maps definable subgroups of $G$ to definable subgroups of $G$. Moreover, if $X$ is also finite then $X$ is  fixed point-wise by $\alpha^k$ for some $k\in \mathbb{N}\setminus\{0\}$.
\end{remark}

\subsection{Finding ${\rm PGL}_2(F)$ inside $P$}\label{subsec:PGL2-in_P} In this section we consider both $\alpha$ and $\alpha^2$ and hence we use subscripts for the subgroups of $P_2$; so ${\rm Soc}(P_2)\cong {\rm PSL}_2(F_2)$. We shall prove that ${\rm PGL}_2(F)\leqslant P$ by studying the fixed points of $\alpha|_{{\rm Soc}(P_2)}$.

\begin{fact}\label{propo:main}Let $K$ be a field of characteristic $\neq 2$ with $[K^\times:(K^\times)^2]=2$, $K\subseteq L$ be a field extension of degree $2$, and $\mathrm{Gal}(L/K)=\{1,f\}$. Then $C_{\mathrm{PSL}_2(L)}(f)\cong \mathrm{PGL}_2(K).$
\end{fact}
\begin{proof}
Let $\delta\in L$ be such that $\delta^2\in K$ and $L=K(\delta)$. Denote by $d$ the element of $\mathrm{PSL}_2(L)$ corresponding to the diagonal matrix $\mathrm{diag}(\delta,\delta^{-1})$. Then $d\in C_{\mathrm{PSL}_2(L)}(f)$ and by direct calculations we obtain:
\begin{enumerate}
  \item $[\langle\mathrm{PSL}_2(K),d\rangle :\mathrm{PSL}_2(K)]=2$;
  \item $C_{\mathrm{PSL}_2(L)}(f)= \langle\mathrm{PSL}_2(K),d\rangle$.
\end{enumerate}
It clearly follows that $C_{\mathrm{PSL}_2(L)}(f)\cong \mathrm{PGL}_2(K)$.
\end{proof}

\begin{lemma}\label{lemma:socle} For all $n\in \mathbb{N}\setminus \{0\}$, we have ${\rm Soc}(P) \leqslant {\rm Soc}(P_n)$.
\end{lemma}
\begin{proof}The simple socle ${\rm Soc}(P_n)$ is the minimal normal subgroup of $P_n$ for any $n\in \mathbb{N}\setminus \{0\}$. So, we only need to check that ${\rm Soc}(P) \cap {\rm Soc}(P_n)\neq 1$. If not, then $${\rm Soc}(P_n){\rm Soc}(P)/{\rm Soc}(P_n) = {\rm Soc}(P)/({\rm Soc}(P_n)\cap {\rm Soc}(P)) \cong  {\rm Soc}(P)\cong {\rm PSL}_2(F).$$ We have a contradiction as $P_n/{\rm Soc}(P_n)$ is abelian by Proposition~\ref{prop:structure-pf}. \end{proof}

\begin{propo}\label{propo:field-autom} There exists a field automorphism $f\in {\rm Aut}(F_2)$ of order $2$ such that $\alpha|_{{\rm Soc}(P_2)}$ is induced by $f$ and $F=\mathrm{Fix}_{F_2}(f)$.
\end{propo}
\begin{proof}Note first that, by the definition, $\alpha$ is non-trivial on $P_2$, thus $\alpha$ is of order $2$ on $P_2$; it easily follows that $\alpha$ is of order $2$ on ${\rm Soc}(P_2)$. As $$\alpha|_{{\rm Soc}(P_2)}\in {\rm Aut}({\rm Soc}(P_2))\cong {\rm PGL}_2(F_2) \rtimes {\rm Aut}(F_2),$$ there are $y\in {\rm PGL}_2(F_2)$ and $f\in  {\rm Aut}(F_2)$ such that $\alpha|_{{\rm Soc}(P_2)}=yf$. We show that $y=1$. Let $H:=(T\cap {\rm Soc}(P))$ and $H_2:=(T_2\cap {\rm Soc}(P_2))$. By Lemma~\ref{lemma:socle}, up to conjugacy, $H \leqslant H_2$. As $\alpha=yf$ is trivial on  $H<H_2$ and $f(H_2)=H_2$, we have $1 < y(f(H))=H \leqslant (T_2)^y$. So $y(T_2)=T_2$ and thus $y$ acts either trivially or by inversion on $T_2$ (Section~\ref{notation}(\ref{PGL2-NT-CT-Ci})); the latter is impossible, as in this case the field automorphism $f$ would act by inversion on $H \cong (F^\times)^2$.

We have the tower of fields $F\subseteq  \mathrm{Fix}_{F_2}(f) \subseteq F_2$ and hence $${\rm Soc}(P)\leqslant \mathrm{PSL}_2(\mathrm{Fix}_{F_2}(f) )\leqslant P=C_{P_2}(\alpha).$$ Since ${\rm Soc}(P) \unlhd P$ and $\mathrm{PSL}_2(\mathrm{Fix}_{F_2}(f) )$ is simple, we get $\mathrm{PSL}_2(F)=\mathrm{PSL}_2(\mathrm{Fix}_{F_2}(f) )$ and $F=\mathrm{Fix}_{F_2}(f)$.\end{proof}

\begin{corollary}\label{corol:PGL2inP2} ${\rm PGL}_2(F) \leqslant P$ and for any $n\in \mathbb{N}\setminus\{0\}$, $T \leqslant T_n$ and $U \leqslant U_n$. So $P={\rm PGL}_2(F) \rtimes A$ where $A\leqslant {\rm Aut}(F)$ is abelian.
\end{corollary}
\begin{proof}
We apply Fact~\ref{propo:main} for $L:=F_2$ and $K:=\mathrm{Fix}_{F_2}(f)$. By Proposition~\ref{propo:field-autom}, we have $K=F$; so ${\rm PGL}_2(F) \leqslant P$. Structure of $P$ now follows from Proposition~\ref{prop:structure-pf}. 

As ${\rm PGL}_2(F) \leqslant {\rm PGL}_2(F_n)$, up to conjugacy, the Borel subgroup $B=U\rtimes T$ of ${\rm PGL}_2(F)$ embeds in the Borel subgroup $B_n=U_n \rtimes T_n$ of ${\rm PGL}_2(F_n)$ as $B_n$ is a maximal solvable subgroup of ${\rm PGL}_2(F_n)$. The claim follows. \end{proof}

\subsection{Structures of $\dt$ and $\du$}\label{subsec:structures} From now on, we freely use Fact~\ref{fact:def_closure} without referring to it.

\begin{lemma}\label{lemma-centraliser-of-T}$C_P(T)=C_{\dt}(\alpha)=T$ and $C^\circ_{G}(\dt)=\dt^\circ$. \end{lemma}

\begin{proof} Note first that since $\dt$ is abelian we have $T\leqslant C_{\dt}(\alpha)\leqslant C_P(T)$. Let $x\in C_P(T)$. By Corollary~\ref{corol:PGL2inP2}, we have $x =yf$ where $y \in {\rm PGL}_2(F)$ and $f \in A \leqslant \mathrm{Aut}(F)$; so $f(T)=T$. Thus $y\in N_{{\rm PGL}_2(F)}(T)=\langle w, T\rangle$ so $y$, and hence $f$, act on $T$ either trivially or by inversion (Section~\ref{notation}(\ref{PGL2-NT-CT-Ci})); the latter is not possible as $f$ is a field automorphism. So $f=1$ and $x\in C_{{\rm PGL}_2(F)}(T)=T$. So, the first equality in the statement follows and we also have $\dt^\circ=\overline{C_P(T)}^\circ$. We get \[\dt^\circ \leqslant C^\circ_{G}(\dt)= \overline{C^\circ_{G}(\dt) \cap P}  \leqslant \overline{C_P(T)}^\circ=\dt^\circ\] by the abelianity of $\dt$ and the definition of $\alpha$.\end{proof}

\begin{lemma}\label{lemma:C_T(U)}The Weyl involution $w$ inverts $\dt$ and $C_{\dt}(\du)=C_{\du}(\dt)=N_{\du}(\dt)=1$.
\end{lemma}

\begin{proof}Let $X= \{ x \in \dt : w x w=x^{-1} \}.$ As $w$ inverts $T$, we have $T \subseteq X$. One easily observes that $X$ is a definable subgroup of $G$. So $X=\dt$ and $w$ inverts $\dt$.

Note that $C_G(\du)=C_G(U)$ and $C_G(\dt)=C_G(T)$. Then, by Lemma~\ref{lemma-centraliser-of-T} we get $C_{\dt}(U) \cap P \leqslant C_{T}(U)=1$. So, we get $$\overline{C^\circ_{\dt}(\du)  \cap P}=C^\circ_{\dt}(\du) =1.$$ Now $C_{\dt}(\du)$ is finite and thus fixed by some power $\alpha^k$. Applying Lemma~\ref{lemma-centraliser-of-T} and Corollary~\ref{corol:PGL2inP2}, we get $C_{\dt}(\du) \leqslant C_{T_k}(U)=1$.

At this point it is clear that $\db=\du \rtimes \dt$; so $N_{\du}(\dt)=C_{\du}(\dt).$ By Lemma~\ref{lemma-centraliser-of-T}, $C_{\du}(\dt)$ is finite and hence, by considering some power $\alpha^r$ and by applying Lemma~\ref{lemma-centraliser-of-T} and Corollary~\ref{corol:PGL2inP2}, we get $C_{\du}(\dt)\leqslant C_{P_r}(T_r)\cap \du=T_r \cap \du \leqslant C_{T_r}(U)=1$. \end{proof}

\begin{lemma}\label{lemma-centraliser-of-U} $\du^\circ=\du$ and $C_{\du}(\alpha)=U$.

\end{lemma}
\begin{proof} We first prove that $\du$ is connected. Let $X= \du^{\circ} \cap P$ and $Y = \du \cap P$. Clearly $X$ and $X \cap U$ are $T$-normal. By $T$-minimality of $U$ (Section~\ref{notation}(\ref{PGL2-minimalU})), either $X \cap U = 1$ or $X \cap U = U$. The former is not possible as $U$ is infinite and $X \cap U = 1$ implies $U\cong UX/X \hookrightarrow Y/X < \infty$. So $\overline{U}\leqslant \overline{X}$. At the same time we have $\overline{X} = \overline{\du^{\circ} \cap P}= \du^{\circ}$ by the definition of $\alpha$.

We move on to check that $C_{\du}(\alpha)=U$: Since $\du$ is an abelian group, we have $C_{\du}(\alpha)\leqslant C_P(U)\leqslant N_P(U).$ Let $x\in C_{\du}(\alpha)$.  Then $x=yf$, where $y\in {\rm PGL}_2(F)$ and $f\in A \leqslant {\rm Aut}(F)$; so $f(U)=U$. Thus $y\in N_{{\rm PGL}_2(F)}(U)=B$, that is, $y=ut$ for some $u \in U$ and $t \in T$. Now $$u^{-1}x=tf\in (C_{\du}(\alpha)\cap N_P(T)) \leqslant N_{\du}(T)=C_{\du}(T)=C_{\du}(\dt).$$ Lemma~\ref{lemma:C_T(U)} now gives us $x\in U$. \end{proof}

\subsection{Killing the group $A$}\label{subsec:P/S-finite}In this section we shall prove that $P={\rm PGL}_2(F)$. The key point is to show that $C_G(i)=\langle w, \dt \rangle$.

\begin{lemma}\label{lemma:norm-abelian}$N^\circ_{G}(\dt)=C^\circ_{G}(\dt^\circ)=C^\circ_{G}(i)$ is an abelian group.
\end{lemma}
\begin{proof}  We first prove that $ N^\circ_P(T)$ is abelian. 
\begin{claim*} 
$ N^\circ_P(T)$ is abelian. 
\end{claim*}\begin{proofclaim}We first observe that $$(N_P(T))' \leqslant N_{{\rm PGL}_2(F)}(T) \ \text{and} \  (N^\circ_P(U))' \leqslant U.$$ If $x \in P $ then $x=yf$ where $y\in {\rm PGL}_2(F)$ and $f\in A$; so $f(U)=U$ and $f(T)=T$. We get $x \in {\rm PGL}_2(F)N_P (U)$ and \[ P/{\rm PGL}_2(F) \cong N_P (U)/(N_P (U) \cap {\rm PGL}_2(F)) = N_P (U)/B.\] Similarly, $P/{\rm PGL}_2(F) \cong N_{P} (T)/N_{{\rm PGL}_2(F)}(T) $. As $P/{\rm PGL}_2(F) \cong A$ is abelian, we have $(N_P(U))' \leqslant B$ and $(N_P(T))' \leqslant N_{{\rm PGL}_2(F)}(T)$. So $\overline{(N^\circ_P(U))'}=\overline{N^\circ_P(U)}'  \unlhd \db$ and hence the group $\overline{N^\circ_P(U)}=\overline{N_P(U)}^\circ$ is solvable. Then $\overline{N^\circ_P(U)}'$ is nilpotent by \cite{Nesin1989}. Note also that $\du \unlhd \overline{N^\circ_P(U)}'$  (Fact~\ref{fact:def_closure} and Lemma~\ref{lemma-centraliser-of-U}) and that $ \overline{N^\circ_P(U)}'$ is $\alpha$-invariant (being a characteristic subgroup of $\overline{N^\circ_P(U)}$). Set $H:=(Z(\overline{N^\circ_P(U)}')\cap \du)^\circ$. Now $H$ is an infinite (\cite[Lemma 5.1]{ABC}) definable connected $\dt$-normal and $\alpha$-invariant subgroup of $\du$. Lemma~\ref{lemma-centraliser-of-U} gives us $H \cap P \leqslant C_{\du}(\alpha) =U$ and $H \cap P$ is infinite by the definition of $\alpha$. Since $H \cap P$ is $T$-normal, we get $H \cap P= U$. So $H = \overline{H \cap P}=\du$ and $\du \leqslant Z(\overline{N^\circ_P(U)}')$. This means that $(\dt \cap \overline{N^\circ_P(U)}' )\leqslant C_{\dt}(\du)=1$ (Lemma~\ref{lemma:C_T(U)}) and hence $\overline{N^\circ_P(U)}'\leqslant \du$. We get $(N^\circ_P(U))' \leqslant U$.

Now, since $P={\rm PGL}_2(F) \rtimes A$, we have $$N_{P}(T)= N_{{\rm PGL}_2(F)}(T)\rtimes A \cong \langle w, T\rangle \rtimes A \,\, {\rm and} \,\, N_{P}(U)=N_{{\rm PGL}_2(F)}(U)\rtimes A \cong B\rtimes A.$$ So $N_{P}(T)\cap N_P(U)\cong  T\rtimes A$ and $ N_P(T)$ is a finite extension of $N_{P}(T)\cap N_P(U)$ and of $N^\circ_P(T) \cap N^\circ_P(U)$. By the observation above we get $$(N^\circ_P(T) \cap N^\circ_P(U))' \leqslant  U \,\, {\rm and} \,\, (N^\circ_P(T) \cap N^\circ_P(U))'  \leqslant N_{{\rm PGL}_2(F)}(T)= \langle w, T\rangle.$$ So $(N^\circ_P(T) \cap N^\circ_P(U))' =1$, that is, $N^\circ_P(T) \cap N^\circ_P(U)$ is an abelian group; and so is $ N^\circ_P(T)$.\end{proofclaim}

It remains to prove that $\overline{N^\circ_P(T)}=N^\circ_{G}(\dt)=C^\circ_{G}(\dt^\circ)=C^\circ_{G}(i)$. Since $i$ is the unique involution of $T$ (Section~\ref{notation}(\ref{PGL2-torus})), we have $N_P(T) \leqslant C_P(i)$. Given $x\in C_P(i)$, we have $x=yf$ where $y\in {\rm PGL}_2(F)$ and $f\in A$. So $f$ fixes $i$ and $y\in C_{{\rm PGL}_2(F)}(i)=N_{{\rm PGL}_2(F)}(T)$ (Section~\ref{notation}(\ref{PGL2-NT-CT-Ci})). Thus $C_P(i)= N_P(T)$ and $$\overline{N^\circ_P(T)}=\overline{C^\circ_P(i)}=\overline{(C^\circ_{G}(i) \cap P)}=C^\circ_{G}(i).$$ One easily observes that $N_P(\dt)= N_P(T)$: It is enough to check $N_P(\dt)\leqslant N_P(T)$. If $x\in N_P(\dt) $ then for any $t\in T$, we have $xtx^{-1} \in \dt \cap P =T$ (Lemma~\ref{lemma-centraliser-of-T}). Note also that $C^\circ_P(T^\circ)=N^\circ_P(T)$ as the abelian group $ N^\circ_P(T)$ contains $T^\circ$ and that $C_P(\dt^\circ)=C_P(T^\circ)$. Now, by the definition of $\alpha$, we get $$C^\circ_{G}(i)=N^\circ_{G}(\dt)=\overline{N^\circ_{G}(\dt) \cap P} = \overline{N^\circ_P(T)}=\overline{C^\circ_P(T^\circ)}=\overline{C^\circ_{G}(\dt^\circ) \cap P}=C^\circ_{G}(\dt^\circ).$$ \end{proof}

\begin{lemma}\label{lemma:centraliser-of-i}$C_{G}(i)=\langle w, \dt \rangle=N_{G}(\dt)$ and $C_{G}(\dt)=\dt=\dt^\circ$. 
\end{lemma}

\begin{proof} We first prove that $i\in \dt^\circ$. Since $i$ is the unique involution of $T$, we know that $C_P(i)=(T \rtimes \langle w \rangle ) \rtimes A$. Also, we have $\overline{C_P(i)}^\circ=C^\circ_G(i)$ which gives $\dt^\circ \overline{A}^\circ=C^\circ_G(i)$.  By Theorem~\ref{fact:centraliser-Sylow-2}, $i$ belongs to some maximal $2$-torus, say $\Sigma$, of $G$ and hence $i\in \Sigma \leqslant C_G(\Sigma)$. Since  $C_G(\Sigma)$ is connected by Theorem~\ref{th:good_torus_centraliser} we get $i \leqslant C_G(\Sigma) \leqslant C^\circ_G(i)$. So $i=ta$ for some $t\in \dt^\circ$ and $a\in \overline{A}^\circ$ with $\alpha(t)\alpha(a)=ta$. Notice then that $T\cap \overline{A}=1$ as $T\cong F^\times$ acts on $U \cong F^+$ as multiplication but $A$, and hence $\overline{A}$, fixes the the unit element $1_F$ of $F$. Now, by Lemma~\ref{lemma-centraliser-of-T}, we have $(\dt\cap \overline{A})^\circ =\overline{(\dt\cap \overline{A})^\circ \cap P}\leqslant \overline{T\cap \overline{A}}=1$. So $\dt\cap \overline{A}$ is finite and point-wise fixed by some power of $\alpha$. Thus, by replacing $\alpha$ with a suitable power, we have $\dt \cap \overline{A} \leqslant T \cap \overline{A}=1$. Since we have $\alpha(a) a^{-1}\in \dt^\circ \cap \overline{A}^\circ$ we get $\alpha(a)=a\in \overline{A}\cap P$. Bu $a$ centralises $T$ (as $i$ does); so $a\in T \cap\overline{A}=1$ and $i=t\in \dt^\circ$.

Let $\Theta$ be the unique Sylow $2$-subgroup of the abelian group $\dt^\circ$. Since $i\in \dt^\circ$, $\Theta$ is infinite and connected by Theorem~\ref{th:connected-solv-Sylow}. So $\Theta\cong (\mathbb{Z}_{2^\infty})^\ell$ for some $\ell \in \mathbb{N}$. Now the set of involutions $I(\Theta)$ of $\Theta$ is finite and thus point-wise fixed by some power $\alpha^k$. By Lemma~\ref{lemma-centraliser-of-T}, $I(\Theta)\leqslant I(T_k)=\{i\}$. Hence $\Theta \cong \mathbb{Z}_{2^\infty}$. As $C_G(\Theta)$ is connected (Theorem~\ref{th:good_torus_centraliser}) and $\Theta$ is abelian, we have $\Theta \leqslant C_G(\Theta) \leqslant C^\circ_G(i)$. So, by abelianity of $C^\circ_G(i)=C^\circ_G(\dt^\circ)$ (Lemma~\ref{lemma:norm-abelian}), $C_G(\Theta)=C^\circ_G(\dt^\circ)$. Since $C_G(\Theta)$ is connected and abelian,   Lemma~\ref{lemma-centraliser-of-T}  gives us $$\dt \leqslant C_G(\dt^\circ) \leqslant C_G(\Theta) \leqslant C^\circ_G(\dt)=\dt^\circ.$$ So $C^\circ_G(i)=\dt$ and ${\rm pr}_2(\dt)=1$. It follows that ${\rm pr}_2(G)=1$ as $C^\circ_G(i)$ contains a maximal $2$-torus of $G$ (Theorem~\ref{fact:centraliser-Sylow-2}).

The finite group $C_G(i)/\dt$ has exponent $2$ by Theorem~\ref{th:centraliser_of_involution}, and, Fact~\ref{fact:centraliser-Sylow-2} tells us that any Sylow $2$-subgroup of $C_G(i)$ is isomorphic to $\Theta\rtimes \langle w \rangle$. So $C_{G}(i)= \langle w, w_1, \ldots, w_n, \dt\rangle,$ where where $w_k$ inverts $\Theta$ for all $k \in \{1,\ldots, n\}$ (note that we have Torsion lifting \cite[Ex. 11 page. 96]{Borovik-Nesin}). The claim easily follows.  \end{proof}

\begin{propo}\label{proposition:P/S} $P={\rm PGL}_2(F)$. \end{propo}

\begin{proof}We have $P={\rm PGL}_2(F)\rtimes A$, where $A\leqslant {\rm Aut}(F)$. So any element of $A$ belongs to $C_G(i)=\langle \dt, w \rangle$ (Lemma~\ref{lemma:centraliser-of-i}) and thus either centralises or inverts $T$. Since no field automorphism act by inversion we get $A=1$.\end{proof}

\subsection{Controlling the $G$-conjugates of $\du$} \label{subsec:normalisers} In this section we prove that the $
G$-conjugates of $\du$ which are disjoint from $\du$ intersect $\db$ trivially.

\begin{lemma}\label{lemma:U-minimal}For $\bar{u} \in \du\setminus \{0\}$, we have $C_{\dt}(\bar{u})=1$. 
\end{lemma}
\begin{proof}We first observe that $\du=\bigoplus_{j=0}^{m-1} \alpha^j(X)$ for some $\db$-minimal subgroup $X$ (i.e. $X$ is infinite, definable and normal subgroup
of $\db$ which is minimal with respect to these properties): The group $\overline{U}$ has no non-trivial, finite, $\overline{T}$-invariant subgroup as one such would be centralised by the connected group $\overline{T}$ (Lemma~\ref{lemma:centraliser-of-i}) contradictory to Lemma~\ref{lemma:C_T(U)}. Now, let $X \leqslant \overline{U}$ be $\db$-minimal and let $m$ be maximal so that $Y = \sum_{j = 0}^{m-1} \alpha^j(X)$ is a direct sum. Then, $\alpha^m(X) \cap Y \neq 0$ is definable and $\dt$-invariant, but also the image of a $\dt$-minimal group. So $\alpha^m(X) \leqslant Y$. Therefore $Y$ is definable, connected, and $\alpha$-invariant; thus $\overline{Y \cap P} = Y$. So $Y \cap P$ is a non-trivial, $\dt\cap P=T$-normal subgroup of $\overline{U} \cap P = U$. We get $Y \cap P = U$ and $Y = \overline{U}$.

We may now prove that there is some $r\in \mathbb{N}\setminus \{0\}$ so that $C_{\dt}(X)$ is $\alpha^r$-invariant: It is well-known that for any $x \in X\setminus\{0\}$, one has $C_{\dt}(x) = C_{\dt}(X)$ (see e.g. \cite[Fact 2.40]{CJ}). Let $x \in X\setminus\{0\}$ and write $\alpha^m(x) = \sum_{j = 0}^{m-1} \alpha^j(x_j)$ as in the notation above. Let $j$ be such that $x_j \neq 0$. Then: $$\alpha^m(C_{\dt}(X)) = \alpha^m(C_{\dt}(x))=C_{\dt}
(\alpha^m(x))\leqslant C_{\dt}(\alpha^j(x_j))=\alpha^j(C_{\dt}(x_j))=\alpha^j(C_{\dt}(X))$$
proving that $C_{\dt}(X)$ is $\alpha^{m-j}$-invariant where $j < m$.

Denote then $X_k=\alpha^k(X)$ for $k\in \{0, \ldots, m-1\}$. By above, after replacing $\alpha$ with a suitable power, we have $C_{\dt}(X_k)=\alpha(C_{\dt}(X_k))$. Now, let $\bar{u} \in \du\setminus \{0\}$. Clearly there is some $0\neq x_k \in X_k$ so that $C_{\dt}(u) \leqslant C_{\dt}(x_k)=C_{\dt}(X_k)$. We have $$C^\circ_{\dt}(X_k)=\overline{C^\circ_{\dt}(X_k)\cap P} = \overline{C^\circ_{T}(X_k)}.$$ By Proposition~\ref{proposition:P/S}, given $1\neq t\in T$, we have $C^\circ_G(t)=\overline{C^\circ_{P}(t)}=\dt$. So $C_{T}(X_k)=1$ as $X_k$ is infinite. This means that $C_{\dt}(X_k)$ is finite and, by replacing $\alpha$ again with a suitable power, we get $C_{\dt}(\bar{u}) \leqslant C_{\dt}(X_k) \leqslant C_{T}(X_k)=1$.\end{proof}

\begin{lemma}\label{lemma:C_G(u)} For $\bar{u} \in \du\setminus \{0\}$, we have $C^\circ_{G}(\bar{u})=\du$.\end{lemma}

\begin{proof} It is enough to show that $C^\circ_{G}(\bar{u}) \leqslant \du$. It is easy to observe that $i$ inverts $\du$ since it inverts $U$ (Section~\ref{notation}(\ref{PGL2-NU-CU})) and hence $i$ is an involutive automorphism of $C^\circ_G(\bar{u})$. By Lemma~\ref{lemma:centraliser-of-i} and Lemma~\ref{lemma:U-minimal}, we have $C^\circ_{G}(i) \cap C^\circ_G(\bar{u})=1$. So $C_{G}(i) \cap C^\circ_G(\bar{u})$ is finite and thus $C^\circ_{G}(\bar{u})$ is abelian by Fact~\ref{fact:involutory_autom}. Now, let $u \in U\setminus \{0\}$. By Proposition~\ref{proposition:P/S}, $C^\circ_{G}(u)=\overline{C^\circ_{P}(u)}=\du$. Since  $C^\circ_{G}(\bar{u})$ is abelian, we get $u\in C^\circ_{G}(\bar{u}) \leqslant C^\circ_{G}(u)=\du$. \end{proof}

A group $L$ is called a Frobenius group if it has a subgroup $1\neq H < L$ with $H^\ell \cap H = 1$ for all $\ell \in L\setminus H$. In this case $H$ is called the Frobenius complement of $L$. The following fact will be used several times in the rest of the paper.

\begin{fact}\label{Frobenius_splits}\begin{enumerate}
\item \emph{(\cite[Lemma 11.21 and Corollary 11.24]{Borovik-Nesin}).} Let $L$ be a Frobenius group of finite Morley rank with a Frobenius complement $H$. Assume that $L$ has a non-trivial normal subgroup disjoint from $H$ and that $H$ is infinite, connected and contains an involution. Then $L$ is connected.
\item\emph{ (\cite[Theorem 11.32]{Borovik-Nesin}).} Let $L=N\rtimes H$ be a solvable Frobenius group of finite Morley rank. If $X \cap N =1 $ for some $X \leqslant L$ then $X$ is conjugate to a subgroup of $H$.\end{enumerate} \end{fact}

\begin{lemma}\label{lemma:norm-U} $N_{G}(\du)=N_{G}(\db)=\db=\du \rtimes \dt$ is a Frobenius group. \end{lemma}

\begin{proof} We have  $$N^\circ_G(\db) =\overline{N^\circ_{G}(\db) \cap P} \leqslant   \overline{N_P(\db)}=\db.$$ The last equality follows by Proposition~\ref{proposition:P/S} together with the observation  that  $N_P(\db) = N_P(B)$ as $\db \cap P = B$ (Lemma~\ref{lemma-centraliser-of-T} and Lemma~\ref{lemma-centraliser-of-U}). Since $\db$ is connected we get $N^\circ_G(\db)=\db$. Similarly, we have $N^\circ_G(\du)=\db$.

Now, let $\bar{t}\in \dt \setminus \{1\}$ and $n \in N_{G}(\db)$. By Lemma~\ref{lemma:U-minimal}, $C_{N^\circ_G(\db)}(\bar{t})=\dt$. So, if there is $1\neq x \in \dt^n \cap \dt$, then $x=t_1^n=t_2$ for some $\bar{t}_1, \bar{t}_2 \in \dt \setminus \{1\}$ and  \[\dt^n = C_{N^\circ_G(\db)}(\bar{t}_1)^n= C_{N^\circ_G(\db)}(\bar{t}_1^n)= C_{N^\circ_G(\db)}(x)= C_{N^\circ_G(\db)}(\bar{t}_2)=\dt.\] So $n \in N_{G}(\dt) \cap N_{G}(\db)=\dt$ (Lemma~\ref{lemma:centraliser-of-i}). Therefore $N_G(\db)$ (similarly $N_G(\du)$) is a Frobenius group with Frobenius complement $\dt$ and hence connected by Fact~\ref{Frobenius_splits}. \end{proof}

\begin{propo}\label{finite_intersection} $\db \cap \du^g=1$ for all $g \in {G}\setminus \db$.\end{propo}

\begin{proof}Notice first that $\du \cap \du^g=1$: if not, then there is $1\neq x \in \du \cap \du^g$ and, by Lemma~\ref{lemma:C_G(u)}, we have $\du=C^\circ_G(x)= \du^g$ with $g\notin N_G(\du)$ (Lemma~\ref{lemma:norm-U}). Let $V=\db \cap \du^g$; so $V \cap \du=1$. Then $V=X^{\bar{b}}$ for some $X \leqslant \dt$ and $\bar{b} \in \db$ (Lemma~\ref{lemma:norm-U} and Fact~\ref{Frobenius_splits}). So, if there is $1\neq v\in V$, then $v={\bar{t}}^{\bar{b}}={\bar{u}}^g$ for some $1\neq \bar{t}\in \dt$ and $1 \neq \bar{u}\in \du$. We get: $\dt^{\bar{b}} \leqslant C^\circ_G({\bar{t}}^{\bar{b}}) = C^\circ_G(v)=C^\circ_G({\bar{u}}^g)=\du^g.$ This contradiction shows that $V=1$. \end{proof}

\subsection{Identification of $G$}\label{subsec:final-identification} Recall that our aim is to invoke Theorem~\ref{th:DN}. The identification of $G$ may be now done using known techniques. To our knowledge such arguments were first used in the pre-print \cite{Jaligot2000} (this pre-print was never published but its context is essentially in \cite[Section 4]{CJ}). Though all the arguments of the rest of our proof can be found in the literature, to keep the text self-contained, we write down the proof of the final identification of $G$.

Below, we repeatedly use the axioms of the Morley rank as well as the following well-known rank computations (for a group of finite Morley rank $H$):\begin{enumerate}[-]
\item Let $H=AB$, with $A$ and $B$ definable. Then ${\rm rk}(H)={\rm rk}(A)+{\rm rk}(B)-{\rm rk}(A \cap B)$ (\cite[Section 4.2. Ex. 18]{Borovik-Nesin}).
\item Let $h \in H$. Then ${\rm rk}(H)={\rm rk}(x^H)+ {\rm rk}(C_H(x))$ (\cite[Ex. 13, Page 67]{Borovik-Nesin}).
\end{enumerate}

\begin{lemma}\label{lemma:generic-inv}${\rm rk}(i^{G})={\rm rk}(i^G \setminus \db)$.
\end{lemma}

\begin{proof} Clearly all Sylow $2$-subgroups of $\db$ are all isomorphic to $\mathbb{Z}_{2^{\infty}}$ (Lemma~\ref{lemma:centraliser-of-i}) and all involutions in $\db$ are conjugate. So $i^{G} \cap \db=i^{\db}.$ Towards a contradiction, assume that ${\rm rk}(i^G\setminus \db) < {\rm rk}(i^G)$, that is, ${\rm rk}(i^{\db})={\rm rk}(i^G \cap \db)= {\rm rk}(i^G)$. Since $C^\circ_{G}(i) < \db$ (Lemma~\ref{lemma:centraliser-of-i}), we get $${\rm rk}(G)={\rm rk}(i^G)+{\rm rk}(C_G(i))={\rm rk}(i^{\db})+{\rm rk}(C_{\db}(i))={\rm rk}(\db).$$ So $G=\db$; a contradiction.\end{proof}

We define the following set for an involution $k \in i^G \setminus \db$: \[T(k)= \{ \bar{b} \in \db : k \bar{b} k = \bar{b}^{-1}\}.\] 
\begin{lemma}\label{lemma:T(k)}
$T(k)$ is an abelian group which is $\db$-conjugate to a subgroup of $\dt$.
\end{lemma}
\begin{proof}Consider the commutator $[t,s]$ for any $t,s\in T(k)$. We have $$\du =\db'\ni [t,s]=[(t^{-1})^k, (s^{-1})^k]=[t^{-1}, s^{-1}]^k\in \du^k.$$ But $\du \cap \du^k=1$ (Lemma~\ref{finite_intersection}). So $[t,s]=1$ and $T(k)$ is an abelian subgroup of $\db$ intersecting $\du$ trivially; the claim follows (Fact~\ref{Frobenius_splits}). \end{proof}

\begin{lemma}\label{lemma:rank_X2} ${\rm rk}(i^G) \leqslant {\rm rk}(\db)$.\end{lemma}

\begin{proof}We first define the following sets: \[X_1=\{k \in i^G \setminus \db : {\rm rk}(T(k)) < {\rm rk}(\dt) \},\] \[X_2=\{k \in i^G\setminus \db : {\rm rk}(T(k)) = {\rm rk}(\dt) \}.\] Now, for $\ell\in \{1,2\}$, let $\sim_\ell$ be an equivalence relations on $X_\ell$ so that for $k_1, k_2 \in X_\ell$, $k_1 \sim_\ell k_2$ if and only if $k_1$ and $k_2$ are in the same coset of $\db$.

Consider the definable projection
\[ p : X_1 \longmapsto X_1/ \sim_1 . \]Let $(X_1)_j = \{ k_1 \in X_1 : {\rm rk}(p^{-1}(p(k_1)))=j \} .$ Clearly $0\leqslant j \leqslant {\rm rk}(\dt)-1$ and $X_1$ can be written as a disjoint union of finitely many $(X_1)_j$'s; so for some $j_0$, ${\rm rk}(X_1)={\rm rk}((X_1)_{j_0})$. We get ${\rm rk}(X_1) \leqslant {\rm rk}(X_1/\sim_1) + j_0$.  Moreover, since we have  $${\rm rk}(\db)+ {\rm rk}(X_1/\sim_1) \leqslant {\rm rk}(G)={\rm rk}(i^G)+{\rm rk}(\dt),$$ we get ${\rm rk}(X_1/\sim_1)\leqslant{\rm rk}(i^G)-{\rm rk}(\du)$ and ${\rm rk}(X_1) \leqslant {\rm rk}(i^G)-{\rm rk}(\du)+j_0.$ But $j_0 < {\rm rk}(\dt) \leqslant {\rm rk}(\du)$ so ${\rm rk}(X_1) < {\rm rk}(i^G)$. Lemma~\ref{lemma:generic-inv} now gives us ${\rm rk}(X_2)={\rm rk}(i^G)$.

Let now $k_1 \in X_2$. Then, by above, $T(k_1)=\dt^{\bar{u}}$ for some unique element $\bar{u} \in \du$. Now consider the map \[\phi: X_2 /\sim_2 \, \,\longrightarrow \du\] sending an element $k_1/ \sim_2$ to the element $\bar{u}$. By Lemma~\ref{lemma:centraliser-of-i}, ${\rm rk}(X_2/\sim_2)= {\rm rk}(\du)$. As ${\rm rk}(X_2)\leqslant {\rm rk}(X_2/\sim_2)+ {\rm rk}(\dt)$ we get ${\rm rk}(i^G)={\rm rk}(X_2)\leqslant {\rm rk}(\db)$.\end{proof}

\begin{lemma}\label{lemma:axiom3c} $G = \db \sqcup \du w\db$.

\end{lemma}
\begin{proof}
Let $g\in G\backslash \db$. Then the map $ \varphi_{g} : \du \times \db \longmapsto \du g \db$ defined by $(\bar{u},\bar{b})\mapsto \bar{u} g \bar{b}$ is a bijection by Lemma~\ref{finite_intersection}. So we have ${\rm rk}(\du g \db)= {\rm rk}(\db)+{\rm rk}(\du) = {\rm rk}(G)$ by Lemma~\ref{lemma:rank_X2}. Since $G$ is connected, the claim follows.
\end{proof}

At this point, it is routine to check that $G$ is a split Zassenhaus group:

\begin{propo}\label{theorem:G_Zassenhaus}
$G$ is a split Zassenhaus group, acting on the set of left cosets of $\db$ in $G$, with a one-point stabiliser $\db$ and a two-point stabiliser $\dt$.
\end{propo}

\begin{proof}The action of $G$ by left multiplication on the coset space $G/\db$ is $2$-transitive by Lemma~\ref{lemma:norm-U} and Lemma~\ref{lemma:axiom3c} and the stabiliser of points $ \db$ and $ w \db$ is $\dt$ by Lemma~\ref{lemma:centraliser-of-i} and Lemma~\ref{lemma:norm-U}. Let $\bar{u} \in \du\setminus \{0\}$. Then the stabiliser of points $\db, w \db, \bar{u} w \db $ is trivial: If $g$ is from the stabiliser, then $g^{\bar{u} } \in \dt^{\bar{u} } \cap \dt = 1$.\end{proof}

We are ready to conclude the proof of Theorem~\ref{main-stringer}.
\begin{proof}[Proof of Theorem~\ref{main-stringer}] As $G$ is a split Zassenhaus group with a two-point stabiliser $\dt$ and $i\in \dt$, Theorem~\ref{th:DN} gives us $G\cong {\rm PGL}_2(K)$ for some algebraically closed field $K$ of characteristic $\neq 2$. \end{proof}

\section{Concluding remarks}\label{sec:concluding} We finish the paper with the following comments.

\subsubsection{Alternative approach}
We now comment on a different approach. If $m|n$, then the inclusion map ${\rm Soc}(P_m)\rightarrow {\rm Soc}(P_n)$ is
induced by a corresponding embedding of the fields
$F_m\rightarrow F_n$, which we may consider as an inclusion. One may expect:
\begin{enumerate}
  \item each extension $F_m\subseteq F_n$ is algebraic;
  \item the increasing union $\mathbb{F}:=\bigcup F_n$ coincides with the algebraic closure of $F$;
  \item the group $\bigcup {\rm Soc}(P_n)$ is an elementary
substructure of $G$.
\end{enumerate}
Unfortunately, it is unclear to us how to show that the aforementioned
union is an algebraically closed field. We do not even know whether
the corresponding field extensions $F_m\subseteq F_n$ are algebraic.
We do not also see any way of showing that
$\bigcup {\rm Soc}(P_n)$ is an elementary substructure of $G$.
However, it still looks like a promising line of research and we plan
to pick it up in future work.

\subsubsection{Connection to {\em ACFA}} Recall from the introduction that ${\rm ACFA}$ denotes the theory of algebraically closed fields with a generic automorphism (studied and axiomatised by Chatzidakis and Hrushovski in \cite{CH1999}). The model theory of ${\rm ACFA}$ is well-understood. In particular, it is known that every difference field embeds in a model of ${\rm ACFA}$ \cite[Theorem (1.1)]{CH1999} and that, given a model $(K_A, \sigma)$ of ${\rm ACFA}$, the fixed-points ${\rm Fix}_{K_A}(\sigma)$ form a pseudofinite field \cite[Proposition (1.2)]{CH1999}.

We proved that $G\cong{\rm PGL}_2(K)$, for an algebraically closed field $K$ of characteristic $\neq 2$ (Theorem~\ref{th:main}). This means that the pair $(K, \alpha)$ embeds in some model $(K_A, \sigma)$ of ${\rm ACFA}$. Moreover, we know that ${\rm Fix}_{K}(\alpha)=F$ is a pseudofinite field. So, it is natural to ask whether the pair $(K, \alpha)$ elementarily embeds in $(K_A, \sigma)$. We finish the paper with this question:

\begin{question}When is the pair $(K, \alpha)$ a model of ${\rm ACFA}$?
\end{question}

\section*{Acknowledgements} We wish to thank Alexandre Borovik for bringing us together to work on this programme, for his  encouragement and for his many helpful suggestions, especially related to finite groups.

The first author wants to thank Adrien Deloro and Katrin Tent for (separate) useful discussions.

The authors had a chance to work with their colleague Adrien Deloro at Oberwolfach Research Institute for Mathematics on a project which is a continuation of the current work. As a result of this meeting, the generalisation of the main result of this paper is stated in a way which will be directly used in the new work.

\bibliographystyle{plain}
\bibliography{Ulla.2021}

\end{document}